\documentclass[12pt,final]{scrartcl} 

\usepackage[utf8]{inputenc}
\usepackage{amsmath, amstext, amsfonts}

\usepackage[thmmarks,amsmath]{ntheorem}
\usepackage[notref,notcite]{showkeys} 
\usepackage[colorlinks,bookmarks]{hyperref}

\usepackage{tikz}

\usepackage{enumitem}

\setlist{itemsep=1pt,parsep=0pt,topsep=3pt,partopsep=0pt}  
\setenumerate{labelindent=4pt, leftmargin=*} 

\def\itm#1{\rm ({#1})} 
\def\itmit#1{\itm{\it #1\,}} 
\def\rom{\itmit{\roman{*}}} 
\def\abc{\itmit{\alph{*}}}

\theoremstyle{break}
\newtheorem{thm}                   {Theorem}
\newtheorem{lem}           [thm] {Lemma}   
\newtheorem{cor}       [thm] {Corollary}   
\newtheorem{prop}     [thm] {Proposition}

\newtheorem{defi}      [thm] {Definition} 
\newtheorem{claim}     [thm] {Claim}

\makeatletter
\newcommand{\openbox}{\leavevmode
  \hbox to.77778em{%
  \hfil\vrule
  \vbox to.675em{\hrule width.6em\vfil\hrule}%
  \vrule\hfil}}

\newcommand{\proofname}{Proof}
\newcounter{proof}%
\newenvironment{proof}[1][]{
  \th@nonumberplain
  \def\theorem@headerfont{\itshape}%
  \normalfont
  \@thm{proof}{proof}{\proofname{} {#1}\unskip.}}%
  {\@endtheorem}
\makeatother

\newcommand{\hide}[1]{}

\newcommand{\parti}[2][k]{\left(#2_{i,j}\right)_{i\in [#1],j\in [r]}}  

\renewcommand{\tilde}{\widetilde}

\newcommand{\field}[1]{\mathbb{#1}}
\newcommand{\N}{\field{N}}

\newcommand{\eps}{\epsilon}
\renewcommand{\epsilon}{\varepsilon}

\renewcommand{\phi}{\varphi}

\DeclareMathOperator{\dcup}{\ensuremath{\mathaccent\cdot \cup}}  
\newcommand{\deff}{\mathrel{\mathop:}=}
\newcommand{\ffed}{\mathrel{=\mathop:}}

\newcommand{\cH}{\mathcal{H}}

\newcommand{\cI}{\mathcal{I}}

\newcommand{\By}[2]{\overset{\mbox{\tiny{#1}}}{#2}}    
\newcommand{\ByRef}[2]{   \By{\eqref{#1}}{#2} }      

\newcommand{\eqByRef}[1]{ \ByRef{#1}{=} }

\newcommand{\geByRef}[1]{ \ByRef{#1}{\ge} }

\DeclareMathOperator{\bw}{bw}

\DeclareMathOperator{\Left}{Left}
\DeclareMathOperator{\Right}{Right}

\newcommand{\EMAIL}[1]{  \textit{E-mail}: \texttt{#1} } 

\title{Spanning embeddings of arrangeable graphs with sublinear bandwidth%
  \thanks{
    AT and AW were partially supported by DFG grant TA 309/2-2.
    The authors are grateful to the Department of Mathematics at the London
    School of Economics and the Women for Math Science-programme of the Zentrum Mathematik at Technische Universit\"at M\"unchen for financially supporting this collaboration.
    The cooperation of the three authors was supported by a joint CAPES-DAAD project (415/ppp-probral/po/D08/11629, Proj. no. 333/09).
 }
}

  \author{
    Julia B\"ottcher\thanks{%
      Department of Mathematics, London School of Economics, Houghton
      Street, London WC2A 2AE, UK
      \EMAIL{j.boettcher@lse.ac.uk}
    }
    \and Anusch Taraz\thanks{%
      Zentrum Mathematik, Technische Universit\"at M\"unchen, 
      Boltzmannstra\ss{}e~3, D-85747 Garching bei M\"unchen, Germany
      \EMAIL{taraz|wuerfl@ma.tum.de}
    }
    \and Andreas W\"urfl\footnotemark[3]
  } 

\begin{document}


\maketitle


\begin{abstract}
  The Bandwidth Theorem of B\"ottcher, Schacht and Taraz [Mathematische 
  Annalen 343 (1), 175--205]
  gives minimum degree conditions for the
  containment of spanning graphs~$H$ with small bandwidth and bounded
  maximum degree.
  We generalise this result to $a$-arrangeable graphs~$H$ with
  $\Delta(H)\le\sqrt{n}/\log n$, where~$n$ is the number of vertices of~$H$.

  Our result implies that sufficiently large $n$-vertex graphs~$G$ with minimum degree 
  at least $(\frac34+\gamma)n$ contain
  almost all planar graphs on~$n$ vertices as subgraphs.
  Using techniques developed by Allen, Brightwell and  Skokan [Combinatorica, to appear] we can also apply
  our methods to show that almost all planar graphs~$H$ have Ramsey number at
  most $12|H|$. We obtain corresponding results for graphs embeddable on
  different orientable surfaces.
\end{abstract}


\section{Introduction}
The existence of spanning subgraphs in dense graphs has been investigated very successfully over the past decades. Its early stages can be traced back to results by Dirac~\cite{Dirac52} in 1952, who showed that a minimum degree of $n/2$ forces a Hamilton cycle in graphs of order $n$,
and Corr\'adi and Hajnal~\cite{CorHaj63} in 1963 as well as Hajnal and Szemer\'edi~\cite{HajSze70} in 1970, who proved that every graph $G$ with $\delta(G)\ge \tfrac{r-1}r n$ must contain a family of $\lfloor n/r\rfloor$ vertex disjoint cliques, each of size $r$. 
The story gained new momentum when, in a series of papers in the 1990s, Koml\'os, Sark\"ozy, and Szemer\'edi established a new methodology which, based on the Regularity Lemma and the Blow-up Lemma, paved the road to a series of results for spanning subgraphs with bounded maximum degree, 
such as powers of Hamilton cycles, trees, $F$-factors, and planar graphs 
(see the survey~\cite{KueOstSurvey} for an excellent overview
of these and related achievements).

During that period, Bollob\'as and Koml\'os~\cite{Komlos99} formulated a
general conjecture which (approximately) included many of the results mentioned above. 
B\"ottcher, Schacht and Taraz proved this conjecture.

\begin{thm}[B\"ottcher, Schacht, Taraz~\cite{BoeSchTar09}] 
\label{thm:BolKom}
For all $r,\Delta\in\N$ and $\gamma>0$, there exist constants $\beta>0$ and $n_0\in\N$ such that for every $n\ge n_0$ the following holds. 
If $H$ is an $r$-chromatic graph on $n$ vertices with $\Delta(H)\le\Delta$ and bandwidth at most $\beta n$ and if $G$ is a graph on $n$ vertices with minimum degree $\delta(G)\ge \left(\tfrac{r-1}r+\gamma\right)n$, then $G$ contains a copy of $H$.
\end{thm}

Here a graph $H$ has \emph{bandwidth} at most $b$ if there exists a
labelling of the vertices by numbers $1,\dots,n$ such that for every edge
$\{i,j\}\in E(H)$ we have $|i-j|\le b$. It is well known that the
restriction on the bandwidth in Theorem~\ref{thm:BolKom} cannot be omitted.
On the other hand, powers of Hamilton cycles and $F$-factors have constant
bandwidth. Moreover, bounded degree planar graphs and more generally any
hereditary class of bounded degree graphs with small separators have
bandwidth at most $O(n/\log n)$ (see~\cite{BPTW10}). Hence a rich class of
graphs~$H$ is covered by Theorem~\ref{thm:BolKom}.

However, a major constraint of this theorem is that it allows only~$H$ with
constant maximum degree. In fact this is also true for most other results on
spanning subgraphs mentioned above.
There are only few exceptions, such as a result by Koml\'os, Sark\"ozy, and
Szemer\'edi~\cite{KSS_trees}, which shows that each sufficiently large graph
with minimum degree at least $(\tfrac12+\gamma)n$ contains all
spanning trees of maximum degree $o(n/\log n)$.

One aim of this paper is to obtain a corresponding embedding result
for a more general class of graphs with unbounded maximum degree. More
precisely, we will generalise Theorem~\ref{thm:BolKom} to graphs with
unbounded maximum degrees.
We focus on arrangeable graphs.

\begin{defi}[$a$-arrangeable]  \label{def:arrangeable}
Let $a$ be an integer. A graph is called $a$-arrangeable if its vertices can be ordered as $(x_1,\dots ,x_n)$ in such a way that 
$\big|N\big(N(x_i)\cap\Right_i))\cap\Left_i\big)\big| \le a$ for each $1 \le i \le n$, where $\Left_i = \{x_1, x_2,\dots, x_i \}$ and $\Right_i = \{x_{i+1}, x_{i+2},\dots, x_n\}$. 
\end{defi}

Arrangeability was introduced by Chen and Schelp~\cite{CheSch93}.
It generalises the concept of bounded maximum degree because graphs
with maximum degree~$\Delta$ are clearly $(\Delta^2-\Delta+1)$-arrangeable,
and stars are $1$-arrangeable.  Moreover several important graph classes
were shown to be constantly arrangeable: planar graphs are
$10$-arrangeable~\cite{KieTro93} (see also~\cite{CheSch93}) and graphs
without a $K_p$-subdivision are $p^8$-arrangeable~\cite{RodTho97}.

Our main result asserts that we can replace the constant maximum degree
bound in Theorem~\ref{thm:BolKom} by $a$-arrangeability and $\Delta(H)\le
\sqrt n/\log n$.

\begin{thm}[The bandwidth theorem for arrangeable graphs] \label{thm:BolKom:arr}
For all $r,a\in\N$ and $\gamma>0$, there exist constants $\beta>0$ and $n_0\in\N$ such that for every $n\ge n_0$ the following holds. 
If $H$ is an $r$-chromatic, $a$-arrangeable graph on $n$ vertices with $\Delta(H)\le \sqrt{n}/\log n$ and bandwidth at most $\beta n$ and if $G$ is a graph on $n$ vertices with minimum degree $\delta(G)\ge \left(\tfrac{r-1}r+\gamma\right)n$, then $G$ contains a copy of $H$.
\end{thm}

The key ingredient for generalising Theorem~\ref{thm:BolKom} to Theorem~\ref{thm:BolKom:arr} is 
a variant of the Blow-up Lemma for arrangeable graphs, obtained recently by 
B\"ottcher, Kohayakawa, Taraz, and Würfl in~\cite{BKTW_blowup} (see
Theorem~\ref{thm:BK:Blow-up:arr:full}). 

\paragraph{Applications.}

We give one direct application of Theorem~\ref{thm:BolKom:arr}
(Corollary~\ref{cor:app:1}), and one application which uses the techniques
needed in the proof of Theorem~\ref{thm:BolKom:arr} (Theorem~\ref{thm:app:2}).
Both applications concern graphs of fixed genus.

Let~$S$ be an orientable surface and denote by $g(S)$ the genus of~$S$.
Let $\cH_S(n)$ be the family of $n$-vertex graphs embeddable on~$S$
and let $\cH_S(n,\Delta)$ be the family of those graphs in $\cH_S(n)$ with
maximum degree at most~$\Delta$. The celebrated Four Colour Theorem~\cite{AppHak77,RobSanSeyRob97} and
the affirmative solution of Heawood's Conjecture~\cite{Hea90,RinYou68} guarantee
that each graph in $\cH_S(n)$ can be coloured with
\begin{equation}
\label{eq:Heawood}
  r(S):= \Big\lfloor\frac{7 + \sqrt{1+48g(S)}}{2}\Big\rfloor
\end{equation}
colours. Moreover, in~\cite{BPTW10} it was shown that graphs in
$H\in\cH_S(n,\Delta)$ have bandwidth at most
\begin{equation}
\label{eq:bwS}
  \bw(S,n,\Delta):=\frac{15n\log\Delta}{\log n - \log\min\big(1,g(S)\big)}\,.
\end{equation}

Hence, as observed in~\cite{BPTW10}, it is a direct consequence of
Theorem~\ref{thm:BolKom} that large $n$-vertex graphs~$G$ with minimum degree
at least $(\frac{r-1}{r}+\gamma)n$ contain all graphs from
$\cH_S(n,\Delta)$ as subgraphs, which extends results of K\"uhn, Osthus and
Taraz~\cite{KuOsTa05} (see also~\cite{KuhOst05}).
With the help of Theorem~\ref{thm:BolKom:arr} we are now able to say
considerably more -- namely, that in fact \emph{almost all} graph from
$\cH_S(n)$ are contained in each such graph~$G$. 

Indeed, McDiarmid and Reed~\cite{McDRee08} proved that
for each fixed~$S$, if we draw a graph~$H$ uniformly at random from
$\cH_S(n)$ then asymptotically almost surely~$H$ has maximum degree of order
\begin{equation}
\label{eq:DeltaS}
  \Delta(S,n):=\Theta_S(\log n)\,.
\end{equation}
Moreover, clearly $K_{r(S)+1}$ cannot be embedded in~$S$ and hence
graphs from $\cH_S(n)$ are $K_{r(S)+1}$-minor free. It thus follows from
the result of R\"odl and Thomas~\cite{RodTho97} mentioned above that the graphs
in $\cH_S(n)$ are $a(S)$-arrangeable with
\begin{equation}
\label{eq:aS}
a(S) := \big( r(S) + 1\big)^8\,.
\end{equation}
In conclusion, using~\eqref{eq:Heawood}, \eqref{eq:bwS}, \eqref{eq:DeltaS}
and~\eqref{eq:aS} we immediately obtain the following corollary of
Theorem~\ref{thm:BolKom:arr}.

\begin{cor} \label{cor:app:1} 
  Let $\gamma>0$, let $S$ be an orientable surface and let $G$ be an $n$-vertex graph
  with $\delta(G)\ge \big(\tfrac{r(S)-1}{r(S)}+\gamma\big)n$.  If~$H$
  is drawn uniformly at random from $\cH_S(n)$, then~$G$ contains~$H$
  almost surely.

  In particular, if $\delta(G)\ge(\frac34+\gamma)n$ then~$G$ contains
  almost all planar graphs on~$n$ vertices.
\end{cor}

Our second application concerns Ramsey numbers of graphs in $\cH_S(n)$. For
a graph~$H$ we denote by $R(H)$ the two-colour Ramsey number of~$H$.
Allen, Brightwell, and Skokan~\cite{AlBrSk10} proved that graphs with
bounded maximum degree and small bandwidth have small Ramsey numbers.

\begin{thm}[Allen, Brightwell and Skokan~\cite{AlBrSk10}] 
  \label{thm:AlBrSk} 
  For all $\Delta\in\N$, there exist constants $\beta>0$ and $n_0$ such
  that for every $n\ge n_0$ the following holds.
  If~$H$ is an $n$-vertex graph with maximum degree at most
  $\Delta$ and $\bw(H)\le\beta n$, then $R(H)\le (2\chi(H)+4)n$.
\end{thm}

With the help of~\eqref{eq:Heawood} and~\eqref{eq:bwS} this implies that
for any fixed orientable surface~$S$ and any fixed $\Delta$ each graph
$H\in\cH_S(n,\Delta)$ satisfies $R(H)\le \big(2 r(S)+4\big)n$ if~$n$ is
sufficiently large. In particular, large planar graphs $H$ with bounded
maximum degree have Ramsey number $R(H)\le 12|H|$.

This together with the fact that planar graphs are known to have at most linear 
Ramsey number (see~\cite{CheSch93})
led  Allen, Brightwell, and Skokan to conjecture that
in fact \emph{all} sufficiently large planar graphs~$H$ have Ramsey number
at most $12|H|$. Combining their methods with ours we can now show that
this is true for \emph{almost all} planar graphs. 

\begin{thm} \label{thm:app:2} 
  Let $S$ be an orientable surface.  If~$H$ is drawn uniformly at random
  from $\cH_S(n)$, then almost surely $R(H)\le (2r(S)+4)n$.

  In particular, for almost every planar graph $H$ we have $R(H) \le
  12|H|$.
\end{thm}

\paragraph{Organisation.}

In Section~\ref{sec:outline} we give an outline of our proof of
Theorem~\ref{thm:BolKom:arr}. This proof builds on partitioning results
for~$G$ and for~$H$, which we present in Section~\ref{sec:Lemmas}, and on a
variant of the  Blow-up Lemma for arrangeable graphs, which we discuss
in Section~\ref{sec:BlowUp}. We then present the actual
proof of Theorem~\ref{thm:BolKom:arr} in Section~\ref{sec:Proof:1}. 
We close with the proof of Theorem~\ref{thm:app:2} in
Section~\ref{sec:Proof:3} and with some concluding remarks in Section~\ref{sec:concl}.

\section{Outline}
\label{sec:outline}

Many of the results concerning the embedding of spanning,
bounded degree graphs follow a general agenda which is nicely
described in the survey paper~\cite{Komlos99} by Koml\'os. This agenda
consists of five main steps:  firstly preparing~$H$, secondly preparing~$G$,
thirdly assigning parts of~$H$ to parts of~$G$, fourthly connecting those
parts, and fifthly embedding the parts of~$H$ separately, via the Blow-up
Lemma.

In the proof of Theorem~\ref{thm:BolKom:arr} we follow a similar agenda.
The preparation for~$G$ uses, as is usual, Szemer\'edi’s Regularity Lemma
and some additional work to produce a suitable partition of~$G$. For this
step we can make use of a lemma from~\cite{BoeSchTar09} (see
Lemma~\ref{lem:forG}).  

The preparation of~$H$  (see Lemma~\ref{lem:forH}) makes use of the bandwidth
of~$H$ and produces a partition of~$H$ which is compatible to the partition
of~$G$ (in this way we implicitly obtain an assignment of the parts of~$H$
to the parts of~$G$). This step is also similar to the methods
used in~\cite{BoeSchTar09} to partition bounded degree graphs~$H$. However,
we need to strengthen this approach because we now deal with graphs~$H$
whose degrees are no longer bounded by a constant. In other words, we need a
slightly different partitioning lemma for~$H$ in order to make this
partition suitable for the Blow-up lemma that we will use in the next step. 

In a final step we use the two partitions obtained to embed~$H$
into~$G$. Our approach here is slightly different from the steps described
by Koml\'os which are usually used (connecting the parts and embedding the
parts of~$H$ separately). We use the Blow-up Lemma for arrangeable
graphs, which was recently established in~\cite{BKTW_blowup}, to formulate
an embedding result (see Theorem~\ref{thm:Blow-up:arr:mixed})
which can handle super-regular and merely regular pairs
simultaneously and make use of a spanning subgraph of the reduced graph of
the partition for~$G$. This enables us to embed~$H$ into~$G$ at once.

\section{Lemmas for $G$ and $H$} \label{sec:Lemmas}

In this section we formulate a partitioning lemma for~$G$, which asserts
that~$G$ has a regular partition suitable for our purposes, and a
corresponding partitioning lemma for~$H$. Both these lemmas are tailored to
the application of the version of the Blow-up Lemma that we will give in the next section.

We first introduce some notation. Let~$G$, $H$ and $R$ be graphs with
vertex sets $V(G)$, $V(H)$, and $V(R)=\{1,\dots,s\}\ffed[s]$. For $v\in
V(G)$ and $S\subseteq V(G)$ we define $N(v,S)\deff N(v)\cap S$. Let
$A,B\subseteq V(G)$ be non-empty and disjoint, and let
$0\le \eps,\delta\le 1$. The \emph{density} of the pair $(A,B)$ is
$d(A,B):=e(A,B)/(|A||B|)$. The pair $(A,B)$ is \emph{$\eps$-regular}, if
$|d(A,B)-d(A',B')|\le\eps$ for all $A' \subseteq A$ and $B' \subseteq B$
with $|A'|\geq\eps|A|$ and $|B'|\geq\eps|B|$. An $\eps$-regular pair
$(A,B)$ is called \emph{$(\eps,\delta)$-regular} if $d(A,B)\ge\delta$, and
\emph{$(\eps,\delta)$-super-regular} if $|N(v,B)| \ge \delta|B|$ for all
$v\in A$ and $|N(v,A)|\ge \delta|A|$ for all $v\in B$.

Let $G$ have the partition $V(G)=V_1\dcup\dots\dcup V_s$ and $H$
have the partition $V(H)=W_1\dcup\dots\dcup W_s$. 
We say that $(V_i)_{i\in[s]}$ is \emph{$(\eps,\delta)$-(super-)regular on $R$} if
$(V_i,V_j)$ is an $(\eps,\delta)$-(super-)regular pair for every $ij\in
E(R)$. In this case~$R$ is also called \emph{reduced graph} of the
(super-)regular partition. The partition classes $V_i$ are also called
\emph{clusters}.

For all $n,k,r\in\N$, we call an integer partition $\parti{n}$ of~$n$ 
$r$-\emph{equitable}, if $|n_{i,j}-n_{i,j'}|\le 1$ for all $i\in[k]$ and $j,j'\in[r]$. Let $B_k^r$ be the $kr$-vertex graph obtained from a path on $k$ vertices by replacing every vertex by a clique of size $r$ and replacing every edge by a complete bipartite graph minus a perfect matching. More precisely, $V(B_k^r)=[k]\times[r]$ and
\begin{equation}
  \{ (i,j),(i',j')\} \in E(B_k^r) \text{\quad iff \quad $i=i'$ or $|i-i'|=1 \wedge j\neq j'$.}
\end{equation}
Let $K_k^r$ be the graph on vertex set $[k]\times[r]$ that is formed by the disjoint union of $k$ complete graphs on $r$ vertices. Obviously, $K_k^r\subseteq B_k^r$.

Now we can formulate the partition lemma for~$G$, which we take from~\cite[Lemma~6]{BoeSchTar09}.

\begin{lem}[Lemma for $G$~\cite{BoeSchTar09}] \label{lem:forG}
For all $r\in\N$ and $\gamma>0$ there exists $d>0$ and $\eps_0>0$ such that for every positive $\eps\le\eps_0$ there exist $K_0$ and $\xi_0>0$ such that for all $n\ge K_0$ and for every graph $G$ on vertex set $[n]$ with $\delta(G)\ge(\tfrac{r-1}r+\gamma)n$ there exist $k\in [K_0]$ and a graph $R_k^r$ on vertex set $[k]\times[r]$ with
\begin{enumerate}[label=\itmit{R\arabic{*}}]
\item\label{lem:forG:R1} $K_k^r\subseteq B_k^r\subseteq R_k^r$,
\item\label{lem:forG:R2} $\delta(R_k^r)\ge(\tfrac{r-1}r+\gamma/2)kr$, and
\item\label{lem:forG:R3} there is an $r$-equitable integer partition $(m_{i,j})_{i\in[k],j\in[r]}$ of $n$ with $(1+\eps)n/(kr)\ge m_{i,j}\ge (1-\eps)n/(kr)$ such that the following holds.\footnote{The upper bound on $m_{i,j}$ is implicit in the proof of Lemma~7 in~\cite{BoeSchTar09}.}
\end{enumerate}
For every partition $(n_{i,j})_{i\in[k],j\in[r]}$ of $n$ with $m_{i,j}-\xi_0n\le n_{i,j}\le m_{i,j}+\xi_0n$ there exists a partition $(V_{i,j})_{i\in[k],j\in[r]}$ of $V$ with
\begin{enumerate}[label=\itmit{G\arabic{*}}]
\item\label{lem:forG:G1} $|V_{i,j}|=n_{i,j}$,
\item\label{lem:forG:G2} $(V_{i,j})_{i\in[k],j\in[r]}$ is $(\eps,d)$-regular on $R_k^r$, and
\item\label{lem:forG:G3} $(V_{i,j})_{i\in[k],j\in[r]}$ is $(\eps,d)$-super-regular on $K_k^r$.
\end{enumerate}
\end{lem}

The remainder of this section is dedicated to a corresponding partitioning
lemma for~$H$, which again will be similar to the Lemma for~$H$
in~\cite{BoeSchTar09} (Lemma~8 in that paper). However, we need to strengthen
the conclusion of this lemma. We shall point out the main differences below.

Again, we start with some definitions.  Let $H$ be a graph on $n$ vertices
and $\sigma: V(H)\to \{0,\dots,r\}$ be a proper $(r+1)$-colouring of $H$. A
set $W\subseteq V(H)$ is called \emph{zero free} if $\sigma^{-1}(0)\cap
W=\emptyset$. Now assume that the vertices of $H$ are labelled $1,\dots,n$
and that this labelling is a labelling of bandwidth at most $\beta n$ for
some $\beta>0$. Given an integer $\ell$, an $(r+1)$-colouring $\sigma:
V(H)\to \{0,\dots,r\}$ of $H$ is said to be $(\ell,\beta)$-\emph{zero free}
with respect to such a labelling if any $\ell$ consecutive blocks contain
at most one block with zeros. Here a block is a set of the form
$B_t\deff\{(t-1)4r\beta n+1,\dots, t4r\beta n\}$, $t=1,\dots, 1/(4r\beta)$.

\begin{lem}[Lemma for $H$] \label{lem:forH}
Let $r,k\ge 1$ be integers and let $\beta,\xi>0$ satisfy $\beta\le \xi^2/(1200r)$.
Let $H$ be a graph on $n$ vertices and assume that $H$ has a labelling of bandwidth at most $\beta n$ and an $(r+1)$-colouring that is $(10/\xi,\beta)$-zero free with respect to this labelling. Let $R_k^r$ be a graph with $V(R_k^r)=[k]\times[r]$ such that 
\begin{enumerate}[label=\itmit{R\arabic{*}$^*$\hspace{-.1em}}]
\item\label{lem:forH:R1} $K_k^r\subseteq B_k^r\subseteq R_k^r$, and
\item\label{lem:forH:R2} for every $i\in[k]$ there is a vertex $s_i \in ([k]\setminus \{i\}) \times [r]$ with $\{s_i, (i,j)\}\in E(R_k^r)$ for every $j \in [r]$.
\end{enumerate}
Furthermore, suppose $(m_{i,j})_{i\in[k],j\in[r]}$ is an $r$-equitable integer partition of $n$ with $m_{i,j}\ge 12\beta n$ for every $i\in[k]$ and $j\in[r]$. Then there exists a mapping $f:V(H)\to [k]\times[r]$ and a set of special vertices $X\subseteq V(H)$ with the following properties, where we set $W_{i,j}\deff f^{-1}(i,j)$.
\begin{enumerate}[label=\itmit{H\arabic{*}}]
\item\label{lem:forH:H1} $|X\cap W_{i,j}|\le \xi n$ and $|N_H(X\cap W_{i,j})\cap W_{i',j'}|\le \xi n$ for all $i,i'\in[k]$, $j,j'\in[r]$,
\item\label{lem:forH:H2} $m_{i,j}-\xi n \le |W_{i,j}|\le m_{i,j}+\xi n$ for every $i\in[k]$ and $j\in[r]$,
\item\label{lem:forH:H3} for every edge $\{u,v\}\in E(H)$ we have $\{f(u),f(v)\}\in E(R_k^r)$, and
\item\label{lem:forH:H4} if $\{u,v\}\in E(H)\setminus E(H[X])$ then $\{f(u),f(v)\}\in E(K_k^r)$.
\end{enumerate}
\end{lem}

This lemma differs from Lemma~8 in~\cite{BoeSchTar09} in that the
conclusion~\ref{lem:forH:H4} is stronger. In order to
obtain this stronger conclusion we had to strengthen the notion of
zero-freeness as well.
Nevertheless the proof of this modified Lemma for~$H$ closely follows
the proof in~\cite{BoeSchTar09}. We use the following propositions.

\begin{prop}[Proposition 20 in~\cite{BoeSchTar09}] \label{prop:EvenOut}
Let $c_1,\dots,c_r$ be such that $c_1\le c_2\le\dots\le c_{r-1}\le c_r\le c_1+x$ and $c_1',\dots,c_r'$ be such that $c_r'\le c_{r-1}'\le\dots\le c_2'\le c_1'\le c_r'+x$. If we set $c_i''\deff c_i+c_i'$ for all $i\in[r]$ then
\[
 \max_i\{c_i''\} \le \min_i\{c_i''\}+x\,.
\]
\end{prop}

\begin{prop}[Proposition 22 in~\cite{BoeSchTar09}] \label{prop:Switching}
Assume that the vertices of $H$ are labelled $1,\dots,n$ with bandwidth at most $\beta n$ with respect to this labelling.  Let $s\in [n]$ and suppose further that $\sigma\colon\,[n] \to \{0,\dots,r\}$ is a proper $(r+1)$-colouring of~$V(H)$ such that $[s-2\beta n,s+2\beta n]$ is zero free.

Then for any two colours $l,l'\in [r]$ the mapping $\sigma'\colon\, [n]\to \{0,\dots,r\}$ defined by
\[
  \sigma'(v):= \begin{cases}
    l  & \text{ if } \sigma(v)=l', s<v \\
    l' & \text{ if } \sigma(v)=l,  s+\beta n<v \\
    0  & \text{ if } \sigma(v)=l, s-\beta n\le v\le s+\beta n \\
    \sigma(v) & \text{ otherwise }
  \end{cases}
\]
is a \emph{proper} $(r+1)$-colouring of $H$.
\end{prop}

By repeatedly applying Proposition~\ref{prop:Switching} we can
transform a colouring of~$H$ into a balanced colouring by allowing some more vertices to be coloured with colour~$0$. This is a first step towards the proof
of Lemma~\ref{lem:forH}.

In order to make this precise we need the following definition. For
$x\in\N$, a colouring $\sigma:[n]\to\{0,\dots,r\}$ is called
$x$-\emph{balanced}, if for each pair $a,b\in [n]\cup \{0\}$ and each
$i\in[r]$, we have
\[ \frac{b-a}r -x \le \left| \sigma^{-1}(i)\cap \{a+1,\dots,b\}\right| \le \frac{b-a}r +x \]
and $|\sigma^{-1}(0)|\le x$.

\begin{prop} \label{prop:Balanced}
Assume that the vertices of $H$ are labelled $1,\dots,n$ with bandwidth at most $\beta n$ and that $H$ has an $(r+1)$-colouring that is $(2\ell,\beta)$-zero free with respect to this labelling. Let $\xi=1/\ell$ and $\beta\le \xi^2/(12r)$. Then there exists a proper $(r+1)$-colouring $\sigma:V(H)\to \{0,\dots,r\}$ that is $(\ell,\beta)$-zero free and $6\xi n$-balanced.
\end{prop}

\begin{proof}
The idea of the proof is to split $H$ into small parts and use Proposition~\ref{prop:Switching} to switch colours in the parts. This allows us to even out differences in the sizes of the colour classes and obtain a balanced colouring.

Recall that the blocks $B_1,\dots,B_{1/4r\beta}$ of $H$ are the vertex sets of the form $B_t=\{(t-1)4r\beta n+1,\dots,t4r\beta n\}$. 

We start by identifying so called \emph{switching blocks}. They will be used to exchange the colours between parts of $H$. With the help of Proposition~\ref{prop:Switching}, which will colour some vertices with 0, we choose the switching blocks in such a way that every $\ell$ consecutive blocks contain at most one block which either had zeros in the original colouring or one switching block. As the ordering of $H$ is $(2\ell,\beta)$-zero free this can be done so that every consecutive $3\ell$ blocks contain at least one switching block.
We next explain how to use the switching blocks.

\begin{claim} \label{cl:recolour}
Let $\sigma:[n]\to \{0,\dots,r\}$ be a proper $(r+1)$-colouring of $H$, $B_t$ a zero free block and $\pi$ any permutation of $[r]$. Then there exists a proper $(r+1)$-colouring $\sigma'$ of $H$ with $\sigma'(v) = \sigma(v)$ for all $v\in \bigcup_{i<t}B_i$ and $\sigma'(v) = \pi(\sigma(v))$ for all $v\in \bigcup_{i>t}B_i$.
\end{claim}

Indeed, every permutation $[r]$ is the concatenation of at most $r$ transpositions, i.e., permutations that exchange only two elements. We split the block $B_t$ into $r$ disjoint intervals of length $4\beta n$ and decompose $\pi$ into at most $r$ transpositions. The claim then follows from Proposition~\ref{prop:Switching}.

\medskip

Let $\{s_1,s_2,\dots,s_p\}$ be the set of indices belonging to switching blocks. For ease of notation let $s_0=0$ and let $s_{p+1}=1/(4r\beta)+1$. Further let $B^*(t)\deff \bigcup_{i\le t} \left( \bigcup_{s_{i-1}\le j<s_i} B_j \right)$, $c_i(t) \deff |\{ v \in B^*(t) : \sigma(v)=i\}|$ and $\tilde{c}_i(t) \deff |\{ v \in B^*(t+1)\setminus B^*(t) : \sigma(v)=i\}|$ for $t\in[p]$.
We inductively construct a proper $(r+1)$-colouring of $H$ with
\begin{align} \label{eq:H:balanced:1}
   \max_{i} \{c_i(t)\} \le \min_{i} \{c_i(t)\} + \xi n
\end{align}
for every $t\in[p+1]$. 

Note that any proper colouring of $H$ satisfies~\eqref{eq:H:balanced:1} for $t=1$ as $|B^*(1)|\le 3\ell 4r\beta n \le \xi n$ because $s_1\le 3\ell$. So let $\sigma$ be a proper $(r+1)$-colouring which satisfies~\eqref{eq:H:balanced:1} for all $t'\le t$. Without loss of generality we assume that $c_1(t)\le c_2(t)\le\dots\le c_r(t)\le c_1(t)+\xi n$. We define the switching for block $t$ to be any permutation $\pi$ which satisfies $\tilde{c}_{\pi(r)}(t)+\xi n\ge \tilde{c}_{\pi(1)}(t)\ge \tilde{c}_{\pi(2)}(t)\ge\dots \ge \tilde{c}_{\pi(r-1)}(t)\ge \tilde{c}_{\pi(r)}(t)$. Such a permutation exists as $|B^*(t+1)\setminus B^*(t)|\le \xi n$. We apply Claim~\ref{cl:recolour} to $\sigma$, the block $B_t$ and the permutation $\pi$ and obtain a new proper $(r+1)$-colouring $\sigma'$. Let $c_i'(t) \deff |\{ v \in B^*(t) : \sigma'(v)=i\}|$. It follows from Proposition~\ref{prop:EvenOut} that $c_i'(t+1) = c_i(t) + \tilde{c}_{\pi(i)}(t)$ satisfies
\begin{align} \label{eq:H:balanced:2} 
  \max_{i} \{c_i'(t+1)\} \le \min_{i} \{c_i'(t+1)\} + \xi n\,.
\end{align}
Therefore, the colouring $\sigma'$ satisfies~\eqref{eq:H:balanced:1} for every $t'\le t+1$. Let $\sigma^*$ be a colouring of $H$ which satisfies~\eqref{eq:H:balanced:1} for every $t\le p+1$. Then $\sigma^*$ is a proper $(r+1)$-colouring and $(\ell,\beta)$-zero free by construction. It remains to show that $\sigma^*$ is also $6\xi n$-balanced.

For this purpose consider any interval $[a,b]\deff \{a,a+1,\dots,b\}\subseteq [n]$ and let $a'=s_t$ for the smallest switching block index $s_t\ge a$. Similarly let $b'$ be the biggest $s_{t'}$ with $s_{t'}\le b$. Fix a colour $i\in[r]$ and let $C_i\deff (\sigma^*)^{-1}(i)$. Clearly $C_i\cap [a,a']$ and $C_i\cap [b',b]$ are of size $\xi n$ at most. Moreover, if $z$ denotes the number of vertices $x$ in $[a',b']$ with $\sigma^*(x)=0$, then 
\[
 \left|C_i\cap [a',b']\right| = \left|C_i\cap [b']\right| -  \left|C_i\cap [a']\right| \eqByRef{eq:H:balanced:2} \frac{b'-a'-z}r \pm 2\xi n\,.
\]
Because $\sigma^*$ is $(\ell,\beta)$-zero free we have $z\le \xi n$. Hence $\big((b-a)-(b'-a')-z\big)/r\le 2\xi n$ implies 
\[
  \left|C_i\cap [a,b] \right| \le \left|C_i\cap [a',b'] \right| \pm 2\xi n = \frac{b'-a'-z}r \pm 4\xi n = \frac{b-a}r \pm 6\xi n\,.
\]
\end{proof}

With the help of Proposition~\ref{prop:Balanced} and an appropriate method
for ``cutting up'' a graph~$H$ with a balanced colouring we can now construct
the homomorphism asserted by Lemma~\ref{lem:forH}.

\begin{proof}[of Lemma~\ref{lem:forH}]
  Given $r,k$ and $\beta$, let $\xi$, $H$ and $R_k^r\!\supseteq\!B_k^r\!
  \supseteq\!K_k^r$ be as required.  Assume without loss of generality that
  the vertices of $R_k^r$ are labelled as induced by this copy of $B_k^r$.
  Assume moreover that the vertices of $H$ are labelled $1,\dots,n$ with
  bandwidth at most $\beta n$ and that $H$ has a $(10/\xi,\beta)$-zero free
  $(r+1)$-colouring with respect to this labelling. Let
  $B_1,\dots,B_{1/(4r\beta)}$ be the corresponding blocks of $H$.  Set
  $\xi'=\xi/10$ and note that $\beta\le \xi^2/(1200r) =
  (\xi')^2/(12r)$. Therefore, by Proposition~\ref{prop:Balanced} with input
  $\beta$, $\ell=1/\xi'$, and $H$, there is an $(\ell,\beta)$-zero free and
  $6\xi'n$-balanced colouring $\sigma:V(H)\to \{0,\dots,r\}$ of~$H$.

Given an $r$-equitable integer partition $(m_{i,j})_{i\in[k],j\in[r]}$ of $n$, set $M_i\deff \sum_{j\in[r]}m_{i,j}$ for $i\in[k]$. Now choose indices $0=t_0\le t_1\le \dots \le t_{k-1}\le t_k=1/(4r\beta)$ such that $B_{t_i}$ and $B_{t_i+1}$ are zero free blocks and 
\begin{align} \label{eq:H:cut}
   \sum_{i'\le t_i} |B_{i'}| \le \sum_{i'\le i} M_{i'} < 12r\beta n+\sum_{i'\le t_i} |B_{i'}|\;.
\end{align}
Indeed, such $t_i$ exist as $\sigma$ is $(\ell,\beta)$-zero free and, in particular, two out of every three consecutive blocks are zero free. Furthermore, the $t_i$ are distinct because $m_{i,j}\ge 12 \beta n$. The last $\beta n$ vertices of the blocks $B_{t_i}$ and the first $\beta n$ vertices of the blocks $B_{t_i+1}$ will be called \emph{boundary vertices} of $H$. Observe that the choice of the $t_i$ implies that boundary vertices are never assigned colour $0$ by $\sigma$.

Using $\sigma$, we will now construct $f:V(H)\to [k]\times [r]$ and $X\subseteq V(H)$. For each $i\in[k]$, and each $v\in \bigcup_{t_{i-1}< i'\le t_i}B_{i'}$ we set
\[
    f(v):= \begin{cases}
      s_i & \text{ if $\sigma(v)=0$, } \\
      (i,\sigma(v)) & \text{ otherwise, }
    \end{cases}
\]
where $s_i$ is the vertex which exists by property~\ref{lem:forH:R2}. Further let
\begin{align*}
    X_1&\deff \bigcup_{v\in \sigma^{-1}(0)}\big(\{v\}\cup N_H(v)\big),\\
    X_2&\deff \big\{v\in V(H)\,:\,v\,\text{ is a boundary vertex}\big\}.
\end{align*}
It remains to show that $f$ and $X\deff X_1\cup X_2$ satisfy properties~\ref{lem:forH:H1}--\ref{lem:forH:H4} of Lemma~\ref{lem:forH}.

Recall that there are $1/(4r\beta)$ many blocks in the $(\ell,\beta)$-zero free colouring $\sigma$. The bandwidth-ordering implies that all vertices from $X_1\cup N(X_1)$ lie in blocks that either contain zeros or that are adjacent to blocks that contain zeros (because $|B_|\ge 4r\beta n$). Hence, at most $(3/\ell)/(4r\beta)+3$ out of $1/(4r\beta)$ blocks contain vertices from $X_1\cup N(X_1)$. 
Furthermore, every $W_{i,j}=f^{-1}(i,j)$ contains at most $\beta n$ boundary vertices and at most $\beta n$ vertices adjacent to boundary vertices. Thus 
\begin{align*}
  |X \cap W_{i,j}| &\le  |X_1| + |X_2 \cap W_{i,j}| \le \left(\frac{3(1/\ell)}{4r\beta}+3\right)4r\beta\, n  + \beta n \\
&\le \frac{4}{4r\ell\beta}4r\beta n + \beta n = \frac{4}{10}\xi n + \beta n\le \xi n\,,\\
\intertext{and}
  |N(X) \cap W_{i,j}| &\le  |N(X_1)| + |N(X_2) \cap W_{i,j}| \le \left(\frac{3(1/\ell)}{4r\beta}+3\right)4r\beta\, n + \beta n \le \xi n
\end{align*}
and property~\ref{lem:forH:H1} holds.

It follows from~\eqref{eq:H:cut} that $M_i - 12r\beta n \le |\bigcup_{t_{i-1}< i'\le t_i}B_{i'}| \le M_i + 12r\beta n$. As $(m_{i,j})_{i\in[k],j\in[r]}$ is an $r$-equitable integer partition of $n$ and $\sigma$ is $6\xi' n$-balanced this implies
\[  
m_{i,j} - \xi n \le  \frac{M_i}r - 12\beta n - 6\xi' n \le |f^{-1}(i,j)| \le \frac{M_i}r + 12\beta n + 6\xi' n \le m_{i,j} + \xi n 
\]
for every $j\in[r]$. Hence property~\ref{lem:forH:H2} is satisfied.

Let $\{u,v\}\in E(H)\setminus E(H[X])$ with $u\notin X$. Since vertices with colour $0$ and their neighbours lie in $X$, we know that therefore $\sigma(u)\neq 0\neq \sigma(v)$. Hence $f(u)=(i,\sigma(u))$ and $f(v)=(i',\sigma(v))$ for some $i,i'\in[r]$. If $i\neq i'$, $u$ and $v$ must both be boundary vertices, which contradicts $u\notin X$. Hence $i=i'$ and property~\ref{lem:forH:H4} follows.

Let $\{u,v\}\in E(H[X])$. As $\sigma$ is a proper $(r+1)$-colouring, $\sigma(u)\neq \sigma(v)$. First assume that $\sigma(u)=0$. Then there is an index $i\in [k]$ such that $f(u)=s_i$ and $f(v)=(i,\sigma(v))$. But $\{s_i,(i,\sigma(v))\}\in E(R_k^r)$ by condition~\ref{lem:forH:R2} and so~\ref{lem:forH:H3} holds in this case. It remains to consider the case $\sigma(u)\neq 0 \neq \sigma(v)$. This implies that both $u,v$ are boundary vertices of different colour. 
Since we started with an ordering of bandwidth at most $\beta n$ we have $f(u)=(i,\sigma(u))$ and $f(v)=(i',\sigma(v))$ with $|i-i'|\le1$. Hence $\{f(u),f(v)\}\in E(B_k^r)\subseteq E(R_k^r)$ by condition~\ref{lem:forH:R1} and so property~\ref{lem:forH:H3} also holds in this case.
\end{proof}

\section{A Blow-up Lemma for arrangeable graphs}  \label{sec:BlowUp}

In this section we provide a Blow-up Lemma type result which we shall apply
to prove Theorem~\ref{thm:BolKom:arr} and Theorem~\ref{thm:app:2}.
This results builds on the following Blow-up Lemma for arrangeable graphs
from~\cite{BKTW_blowup}.

\begin{thm}[Arrangeable Blow-up Lemma, full version~\cite{BKTW_blowup}] \label{thm:BK:Blow-up:arr:full}
  For all $C,a,\Delta_R,\kappa \in \N$ and for all $\delta',c>0$ there exist
  $\eps',\alpha'>0$ such that for every integer $s$ there is $n_0$ such that
  the following is true for every $n\ge n_0$.
  Assume that we are given 
 \begin{enumerate}[label=\rom]
    \item\label{item:Blow-up:R} a graph~$R$ on vertex set $[s]$ with $\Delta(R)<\Delta_R$,
    \item\label{item:Blow-up:H} an $a$-arrangeable $n$-vertex graph $H$
      with maximum degree $\Delta(H)\le \sqrt{n}/\log n$, together with a
      partition $V(H)=W_1\dcup\dots\dcup W_s$ such that $uv\in E(H)$ implies $u\in W_i$ and $v\in W_j$ with $ij\in E(R)$,
    \item\label{item:Blow-up:G} a graph $G$ with a partition $V(G)=V_1\dcup\dots\dcup V_s$ that is
      $(\eps',\delta')$-super-regular on $R$ and has $|W_i|\le|V_i|=:n_i$ and $n_i\le \kappa\cdot n_j$ for all $i,j\in[s]$,
    \item for every $i\in [s]$ a set $S_i\subseteq W_i$ of at most
      $|S_i|\le \alpha n_i$ \emph{image restricted vertices}, such that
      $|N_H(S_i)\cap W_j|\le \alpha n_j$ for all $ij\in
      E(R)$,
    \item\label{item:Blow-up:ir} and for every $i\in[s]$ a family
      $\cI_i=\{I_{i,1},\dots,I_{i,C}\}\subseteq 2^{V_i}$ of permissible \emph{image
      restrictions}, of size at least $|I_{i,j}|\ge c n_i$ each, together
      with a mapping $I\colon S_i\to \cI_i$, which assigns a permissible
      image restriction to each image restricted vertex.
  \end{enumerate}
  Then there exists an embedding $\phi\colon V(H) \to V(G)$ such that
  $\phi(W_i)=V_i$ and $\phi(x) \in I(x)$ for every $i\in[s]$ and every
  $x\in S_i$.
\end{thm}

This theorem requires super-regularity for all pairs used in the
embedding. However, in applications this can usually not be guaranteed:
Lemma~\ref{lem:forG} for example provides a partition of $G$ where we know
only for very few regular pairs that they are also super-regular.

The standard approach to deal with a situation like this is to apply the
Blow-up Lemma only \emph{locally} to small groups of clusters where
super-regularity is guaranteed (such as the $K_r$-copies within $K_k^r$ in
Lemma~\ref{lem:forG}) and to use image restrictions to connect these local
embeddings into an embedding of the whole graph~$H$.

Instead, here we combine Theorem~\ref{thm:BK:Blow-up:arr:full} with a
randomisation step in order to obtain the following version of the Blow-up
Lemma for arrangeable graphs that can handle super-regular pairs and merely
regular pairs at once.

This result will allow us to embed a spanning graph~$H$ at once by imposing
the additional restriction that edges which are embedded into pairs that
are regular but not necessarily super-regular are confined to a small
subpair in this pair.

\begin{thm}[Arrangeable Blow-up Lemma, mixed version] 
\label{thm:Blow-up:arr:mixed}
For all $a,\Delta_R,\kappa$ and for all $\delta>0$ there exist $\eps,\alpha>0$ such that for every $s$ there is $n_0$ such that the following is true for every $n_1,\dots,n_s$ with $n_0\le n=\sum n_i$ and $n_i\le\kappa \cdot n_j$ for all $i,j\in[s]$. Assume that we are given graphs $R$, $R^*$ with $V(R)=[s]$, $\Delta(R)<\Delta_R$ and $R^*\subseteq R$, and graphs $G$, $H$ on $V(G)=V_1\dcup \dots \dcup V_s$, $V(H)=W_1\dcup\dots\dcup W_s$ with
\begin{enumerate}[label=\itmit{G\arabic{*}}]
\item\label{thm:BU:G1} $|V_i|=n_i$ for every $i\in[s]$,
\item\label{thm:BU:G2} $(V_i)_{i\in[s]}$ is $(\eps,\delta)$-regular on $R$, and 
\item\label{thm:BU:G3} $(V_i)_{i\in[s]}$ is $(\eps,\delta)$-super-regular on $R^*$.
\end{enumerate}
Further let $H$ be $a$-arrangeable, $\Delta(H)\le \sqrt{n}/\log n$, and let there be a function $f:V(H)\to [s]$ and a set $X\subseteq V(H)$ with 
\begin{enumerate}[label=\itmit{H\arabic{*}}]
\item\label{thm:BU:H1} $|X\cap W_i|\le \alpha n_i$ and $|N_H(X\cap W_i)\cap W_j|\le \alpha n_j$ for every $i\in[s]$ and every $ij\in E(R)$,
\item\label{thm:BU:H2} $|W_i|\le n_i$ for every $i\in[s]$,
\item\label{thm:BU:H3} for every edge $\{u,v\}\in E(H)$ we have $\{f(u),f(v)\}\in E(R)$,
\item\label{thm:BU:H4} for every edge $\{u,v\}\in E(H)\setminus E(H[X])$ we have $\{f(u),f(v)\}\in E(R^*)$.
\end{enumerate}
 Then $H\subseteq G$.
\end{thm}

The idea of the proof is as follows. If $R=R^*$, that is, if all edges in
$R$ correspond to super-regular pairs in $G$, we are done by
Theorem~\ref{thm:BK:Blow-up:arr:full}. 
In general of course this will not be the case. However, we will
artificially create a situation like that:
we carefully construct an auxiliary graph
$G'\supseteq G$ which also has~$R$ as a reduced graph, but which has super-regular
pairs for \emph{all} edges in $R$. We then use Theorem~\ref{thm:BK:Blow-up:arr:full}
to embed~$H$ into~$G'$.
It will then remain to show that we constructed~$G'$ (and the image
restrictions used in the application of Theorem~\ref{thm:BK:Blow-up:arr:full})
sufficiently carefully that this embedding in fact uses only edges from~$G$.

\begin{proof}[of Theorem~\ref{thm:Blow-up:arr:mixed}]
Let $a, \Delta_R, \kappa$ and $\delta>0$ be given. Let $\eps',\alpha'>0$ as in Theorem~\ref{thm:BK:Blow-up:arr:full} with $C\deff 1,a,\Delta_R,\kappa,\delta'\deff\delta/2$, and $c\deff 1/2$ and set $\eps\deff \min\{\eps'/2,1/(2\Delta_R),\delta/2\}$, $\alpha\deff\alpha'$. Let $s$ be given and choose $n_0$ as given by Theorem~\ref{thm:BK:Blow-up:arr:full}. Now let $R,R^*,G,H$ have the required properties. In particular let $V(G)=V_1\dcup \dots \dcup V_s$, $V(H)=W_1\dcup\dots\dcup W_s$ be partitions such that $(V_i)_{i\in[s]}$ is $(\eps,\delta)$-regular on $R$ and $(\eps,\delta)$-super-regular on $R^*$.

For $i\in [s]$ define $U_i$ to be the set of all vertices $v\in V_i$ with $|N_G(v)\cap V_j|\ge (\delta-\eps) n_j$ for all $j\in N_R(i)$. Since $\Delta(R)< \Delta_R$ and all pairs $(V_i,V_j)$ with $j\in N_R(i)$ are $(\eps,\delta)$-regular we have
\begin{align} \label{eq:mixed:restrictions}
  |U_i| \ge |V_i| - \Delta_R \eps |V_i| \ge \tfrac 12 |V_i| \;.
\end{align}

In the next step we construct a graph $G'$ which is super-regular on all pairs $(V_i,V_j)$ with $ij\in E(R)$. For every $ij \in E(R)\setminus E(R^*)$ we do the following. For every vertex $v\in V_i$ with $|N_G(v)\cap V_j|< (\delta-\eps) n_j$ we add edges to $\delta n_j$ randomly selected vertices in $V_j$ thus ensuring the minimum degree for $v$ in $V_j$. 
Let $G'$ be the resulting graph. With positive probability, all pairs $(V_i,V_j)$ with $ij\in E(R)$ are now $(2\eps,\delta-\eps)$-super-regular in $G'$. In particular, there exists at least one graph $G'$ with $(V_i,V_j)$ being an $(2\eps,\delta-\eps)$-super-regular pair in $G'$ for every $ij\in E(R)$ and 
\begin{align}\label{eq:mixed:subsets:1}
G[V_i\cup V_j]&=G'[V_i\cup V_j] &\text{if $ij\in E(R^*)$,}\\ \label{eq:mixed:subsets:2}
G[U_i\cup U_j]&=G'[U_i\cup U_j] &\text{if $ij\in E(R)$.\hspace{.15cm}}
\end{align}

As $G'$ is $(\eps',\delta')$-super-regular for every $ij\in E(R)$ we have
$H\subseteq G'$ by Theorem~\ref{thm:BK:Blow-up:arr:full} even if, for every
$i\in[s]$, we restrict the embedding of vertices in $S_i\deff W_i\cap X$ to
$U_i\in \cI_i\deff \{U_i\}$. This is possible
by~\eqref{eq:mixed:restrictions} and the fact that $|W_i\cap X|\le \alpha
n_i$ and $|N_H(W_i\cap X)\cap W_j|\le \alpha n_j$ for all $i\in[s]$ and all
$ij\in E(R)$.

Moreover, every $uv \in E(H)\cap W_i\times W_j$ with $ij\in E(R)\setminus E(R^*)$ has $u,v\in X$. Therefore, the embedding of $H$ into $G'$ also is an embedding of $H$ into $G$ by~\eqref{eq:mixed:subsets:1} and~\eqref{eq:mixed:subsets:2}.
\end{proof}

\section{Proof of Theorem~\ref{thm:BolKom:arr}}  \label{sec:Proof:1}

\renewcommand{\labelenumi}{(R\arabic{enumi})}

Our strategy for this proof is as follows. We use the Lemma for $G$
(Lemma~\ref{lem:forG}) and the Lemma for $H$ (Lemma~\ref{lem:forH}) to get
a partition of $H$ and a matching regular partition of $G$ which is
$(\eps,\delta)$-(super-)regular wherever edges of $H$ are to be
embedded. Given these partitions, the Blow-up Lemma
(Theorem~\ref{thm:Blow-up:arr:mixed}) guarantees an embedding of $H$ into $G$.

\begin{proof}[of Theorem~\ref{thm:BolKom:arr}]
We first set up the constants. Given $r$, $a$, $\gamma>0$, let $d$, $\eps_0$ be given by Lemma~\ref{lem:forG}. Set $\Delta_R\deff 3r+1/\gamma+1$, $\kappa\deff 2$ and $\delta\deff d$ and let $\eps_{T.\ref{thm:Blow-up:arr:mixed}}$ and $0<\alpha\le 1$ be given by Theorem~\ref{thm:Blow-up:arr:mixed}. Plug this $\eps \deff \min\{\eps_0,1/4,\eps_{T.\ref{thm:Blow-up:arr:mixed}}\}$ into Lemma~\ref{lem:forG} and obtain $K_0,\xi_0$. If necessary decrease $\xi_0$ such that $\xi_0\le \alpha/(2 r K_0)$. 
Choose $\beta,\xi$ such that $\xi \le \xi_0$ and $\beta \le \xi^2/(1200r)$.
Finally for every $s\le r\cdot K_0$ let $n_0$ be sufficiently large for the application of Theorem~\ref{thm:Blow-up:arr:mixed}.

Now let $G$ be any graph on $n\ge n_0$ vertices with $\delta(G)\ge (\tfrac{r-1}r+\gamma)n$. Then Lemma~\ref{lem:forG} returns a $k\le K_0$ and a graph $\tilde{R}_k^r$ on vertex set $[k]\times[r]$ and an $r$-equitable integer partition $(m_{i,j})_{i\in[k],j\in[r]}$ with properties~\ref{lem:forG:R1}--\ref{lem:forG:R3}. In particular, 
\begin{align*}
m_{i,j}&\ge \frac{n}{2kr}\ge \frac{n}{2k}\frac{2K_0\xi_0}{\alpha}\ge \xi n \ge \sqrt{1200r\beta}n\ge 12\beta n
\end{align*}
for all $i\in[k]$, $j\in[r]$.

With this integer partition we return to Lemma~\ref{lem:forH}. Let $H$ satisfy the conditions of Theorem~\ref{thm:BolKom:arr}, in particular $H$ is $r$-chromatic and has bandwidth at most $\beta n$. Hence, clearly there is a labelling of bandwidth at most $\beta n$ with a $(10/\xi,\beta)$-zero free $(r+1)$-colouring. Furthermore, we need to show that there is a graph $R_k^r$ with $B_k^r\subseteq R_k^r\subseteq \tilde{R}_k^r$ which satisfies conditions~\ref{lem:forH:R1} and~\ref{lem:forH:R2} of Lemma~\ref{lem:forH} and additionally has $\Delta(R_k^r)< \Delta_R$. Indeed, $R_k^r$ can be obtained as follows. Recall that $\delta(\tilde{R}_k^r)\ge \big(\tfrac{r-1}r + \gamma/2\big)kr$ by property~\ref{lem:forG:R2}. Thus for every $i\in[k]$ there are at least $\tfrac\gamma2 kr$ vertices $v\in ([k]\setminus\{i\})\times [r]$ with $\{v,(i,j)\}\in E(\tilde{R}_k^r)$ for all $j\in [r]$. 
We say that such a vertex $v$ covers $i$. Now, consecutively choose for each $i=1,\dots,k$ a vertex $v_i\in [k]\times [r]$ among those vertices covering $i$ which has been used as $v_{i'}$ as few times as possible for $i'<i$. Then the edges of $R_k^r$ only consist of edges of $B_k^r$ in $\tilde{R}_k^r$ and all edges $\{v_i,(i,j)\in E(\tilde{R}_k^r)$. Since $\Delta(B_k^r) \le 3r$ we have by the choice of the $v_i$ that
$\Delta(R_k^r) \le 3r+2/\gamma<\Delta_R$. Hence $R_k^r$ satisfies conditions~\ref{lem:forH:R1} and~\ref{lem:forH:R2} of Lemma~\ref{lem:forH}. 

As $r,k,\beta,\xi$ and $R_k^r$ and the $r$-equitable integer partition $(m_{i,j})_{i\in[k],j\in[r]}$ satisfy the requirements of Lemma~\ref{lem:forH}, we obtain a mapping $f:V(H) \to [k]\times[r]$ and a set $X$ which satisfy conditions~\ref{lem:forH:H1}--\ref{lem:forH:H4}. In the next step we will partition $V(G)$ into $\parti{V}$. A vertex $x\in V(H)$ is then embedded into $V_{i,j}\subseteq V(G)$ if and only if $x\in f^{-1}(i,j)$.

Define $n_{i,j}\deff |f^{-1}(i,j)|$ and note that $m_{i,j}-\xi_0 n \le n_{i,j}\le m_{i,j}+\xi_0 n$ by property~\ref{thm:BU:H2}. Thus there exists a partition of $V(G)$ into $(V_{i,j})_{i\in[k],j\in[r]}$ with properties~\ref{lem:forG:G1}--\ref{lem:forG:G3} by Lemma~\ref{lem:forG}. Moreover, $n_{i,j}\le 2 n_{i',j'}$ for all $i,i'\in[k]$ and $j,j'\in[r]$ by property~\ref{lem:forG:R3} and property~\ref{thm:BU:H2} as
\[
  n_{i,j} \le m_{i,j}+\xi_0 n \le (1+\eps)\frac{n}{kr} + \xi_0 n \le 2\left((1-\eps)\frac{n}{kr}-\xi_0 n\right) \le 2\left(m_{i',j'}-\xi_0n\right) \le 2 n_{i',j'}\,.
\]

Now all conditions of Theorem~\ref{thm:Blow-up:arr:mixed} are satisfied and thus $H\subseteq G$.
\end{proof}

\section{Proof of Theorem~\ref{thm:app:2}} \label{sec:Proof:3}

The proof of Theorem~\ref{thm:app:2} closely follows the methods of Allen,
Brightwell and Skokan~\cite{AlBrSk10}.  The restriction on $\Delta(H)$ in
their result (Theorem~\ref{thm:AlBrSk}) originates from the embedding
result they use (Theorem 24 in~\cite{AlBrSk10}).  This embedding result in
turn relies on the Blow-up Lemma and the Lemma for~$H$
in~\cite{BoeSchTar09}.  The following Lemma~\ref{lem:forH:2} is a
consequence of our Lemma for~$H$ (Lemma~\ref{lem:forH}). We shall use this
lemma together with the Blow-up Lemma for arrangeable graphs
(Theorem~\ref{thm:BK:Blow-up:arr:full}) to extend the result of Allen,
Brightwell and Skokan to arrangeable graphs.
 
We denote by $P_m^r$ the $r$-th power of a path $P_m$, that is, $P_m^r$ has vertex set $[m]$ and edge set $\{uv: |u-v|\le r\}$.
Analogously, $C_m^r$ is the $r$-th power of the cycle $C_m$.
 
\begin{lem*} \label{lem:forH:2} 
  For any $\xi>0$ and for any natural numbers $r',m_0$ there
  exists $\beta>0$ such that the following is true.
  Let $H$ be a graph on $n$ vertices that 
  is $r$-colourable for $r\le r'$
  and has $\bw(H)\le \beta n$.  Then for any~$m$ with $2r\le m\le m_0$
  there exists a homomorphism $f: H \to C_m^r$ with $|f^{-1}(i)|\le \tfrac
  nm(1+\xi)$ for every $i\in[m]$.
\end{lem*}
\begin{proof}
  Let $\xi>0$ and $r',m_0$ be given.  We choose $k'$ sufficiently large so
  that $m_0/k'\le\xi/3$ and so that $(k'+r'-1)/m$ is integer for each
  $m\in[m_0]$ and $r\in[r']$. We set
  \begin{equation*}
     \xi'\deff\frac{\xi}{3k'r'} \qquad\text{and}\qquad
     \beta\deff\min\Big(\frac{\xi'^2}{1200r'}\,,\, \frac\xi{6k'r'}\Big)\,.
  \end{equation*}
  Assume that $H$ satisfies the requirements of the lemma. Observe
  that by the definition of~$\beta$ we can assume that the number of
  vertices~$n$ of~$H$ satisfies $n\ge 6k'r'/\xi\ge 6k'r/\xi$ and hence
  \begin{equation}
    \label{eq:forH:xi}
    1+\xi'n
    =\frac{n}{k'r}\Big(\frac{k'r}{n}+k'r\xi'\Big)
    =\frac{n}{k'r}\Big(\frac{k'r}{n}+\frac{\xi}{3}\Big)
    \le\frac{n}{k'r}\cdot\frac{\xi}{2}\,.
  \end{equation}
  Let $m$ with $2r\le m\le m_0$ be given.

  We would now like to start by applying Lemma~\ref{lem:forH} with parameters
  $r,k'$ and $\beta$, $\xi'$.  For this purpose let $R_{k'}^r$ be the graph
  obtained from $B_{k'}^r$ (defined in the beginning of Section~\ref{sec:Lemmas}) by adding all edges of the form
  $\{(i,j),(i+1,j)\}$ where $i\in[k'-1]$ and $i-j \equiv 0 \mod r$ (see Figure~\ref{fig:mapping:1}). These
  additional edges ensure that for every $i\in[k']$ there is a vertex
  $s_i=(i+1,i')$ or $s_i=(i-1,i')$ (where $i'\in[r]$ satisfies $i-i'\equiv 0 \mod r$) such
  that $\{s_i,(i,j)\}\in E(R_{k'}^r)$ for all $j\in[r]$.  Hence the graph
  $R_{k'}^r$ satisfies conditions~\ref{lem:forH:R1} and~\ref{lem:forH:R2}
  of Lemma~\ref{lem:forH}.

  Furthermore let $\lfloor
  n/(k'r)\rfloor\ffed m_{1,1}\le m_{1,2}\le \dots \le m_{k',r}\deff \lceil
  n/(k'r)\rceil$. Then Lemma~\ref{lem:forH} guarantees
  a mapping $f':V(H)\to [k']\times [r]$
  and a set $X\subseteq V(H)$ with properties
  \ref{lem:forH:H1}--\ref{lem:forH:H4}. In the
  following we call each set $f'^{-1}(i,j)$ with $i\in[k']$, $j\in[r]$ an
  \emph{$f'$-class} and use these classes to define a homomorphism $f:V(H)\to
  C_m^r$ with the properties promised by Lemma~\ref{lem:forH:2}.

  We will construct~$f$ in two further steps.  Recall that
  $V(R_{k'}^r)=[k']\times[r]$ and consider the $r$-th power of a path
  $P_{k'+r-1}^r$ on vertex set $V(P_{k'+r-1}^r)=[k'+r-1]$.  First we now define
  a mapping $f^*\colon [k']\times[r] \to [k'+r-1]$ whose purpose is to
  group the $f'$-classes and which is a homomorphism from $R_{k'}^r$ to
  $P_{k'+r-1}^r$. Let $(i,j)\in[k']\times[r]$. Observe that there are
  unique positive integers $\ell$ and $x$ such that $x\in[r]$ and $i=-(r-j) +
  r \cdot \ell + x$. Then set $f^*(i,j):=r \cdot \ell + j$ (see also
  Figure~\ref{fig:mapping:1}).
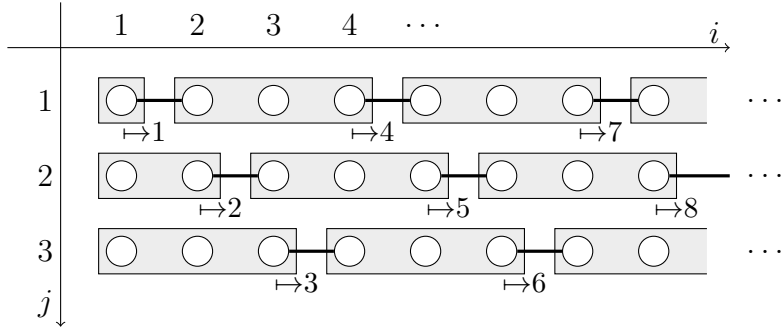
\begin{figure}[t]
\begin{center}
\tikzstyle{cluster}=[circle,draw,fill=white,inner sep=4pt]
\begin{tikzpicture}[scale=1,inner sep=-1pt]
\draw[->] (.2,.3) -- (.2,-4);
\node (l0) at (.0,-3.7) {$j$};
\node (l1) at (0,-1) {$1$};
\node (l2) at (0,-2) {$2$};
\node (l3) at (0,-3) {$3$};

\draw[->] (-.5,-.3) -- (9,-.3);
\node (r0) at (8.8,-.1) {$i$};
\node (r1) at (1,0) {$1$};
\node (r2) at (2,0) {$2$};
\node (r3) at (3,0) {$3$};
\node (r3) at (4,0) {$4$};
\node (r3) at (5,0) {$\dots$};

\node (c1) at (9.5,-1) {$\dots$};
\node (c1) at (9.5,-2) {$\dots$};
\node (c1) at (9.5,-3) {$\dots$};

\draw[fill=black!7] (.7,-.7) rectangle (1.3,-1.3) node [label=270:\small{$\mapsto$1}] {}
      (.7,-1.7) rectangle (2.3,-2.3) node [label=270:\small{$\mapsto$2}] {}
      (.7,-2.7) rectangle (3.3,-3.3) node [label=270:\small{$\mapsto$3}] {}
      (1.7,-.7) rectangle (4.3,-1.3) node [label=270:\small{$\mapsto$4}] {}
      (2.7,-1.7) rectangle (5.3,-2.3) node [label=270:\small{$\mapsto$5}] {}
      (3.7,-2.7) rectangle (6.3,-3.3) node [label=270:\small{$\mapsto$6}] {}
      (4.7,-.7) rectangle (7.3,-1.3) node [label=270:\small{$\mapsto$7}] {}
      (5.7,-1.7) rectangle (8.3,-2.3) node [label=270:\small{$\mapsto$8}] {};
\draw[fill=black!7] (8.7,-.7) -- (7.7,-.7) -- (7.7,-1.3) -- (8.7,-1.3)
      (8.7,-2.7) -- (6.7,-2.7) -- (6.7,-3.3) -- (8.7,-3.3);

\node (11) [cluster] at (1,-1) {};
\node (12) [cluster] at (1,-2) {};
\node (13) [cluster] at (1,-3) {};

\node (21) [cluster] at (2,-1) {};
\node (22) [cluster] at (2,-2) {};
\node (23) [cluster] at (2,-3) {};

\node (31) [cluster] at (3,-1) {};
\node (32) [cluster] at (3,-2) {};
\node (33) [cluster] at (3,-3) {};

\node (41) [cluster] at (4,-1) {};
\node (42) [cluster] at (4,-2) {};
\node (43) [cluster] at (4,-3) {};

\node (51) [cluster] at (5,-1) {};
\node (52) [cluster] at (5,-2) {};
\node (53) [cluster] at (5,-3) {};

\node (61) [cluster] at (6,-1) {};
\node (62) [cluster] at (6,-2) {};
\node (63) [cluster] at (6,-3) {};

\node (71) [cluster] at (7,-1) {};
\node (72) [cluster] at (7,-2) {};
\node (73) [cluster] at (7,-3) {};

\node (81) [cluster] at (8,-1) {};
\node (82) [cluster] at (8,-2) {};
\node (83) [cluster] at (8,-3) {};

\draw[very thick] (11) -- (21)
      (22) -- (32)
      (33) -- (43)
      (41) -- (51)
      (52) -- (62)
      (63) -- (73)
      (71) -- (81)
      (82) -- (9,-2);
\end{tikzpicture}
\caption{An illustration of $f^*$ for $r=3$. The white circle in column~$i$
  and row~$j$ represents the set $f'^{-1}(i,j)$.  The
  homomorphism $f^*$ groups these sets as indicated.
  The thick vertical edges
  indicate the additional edges $\{(i,j),s_i\}$ of the reduced graph $R^r_{k'}$.
  For example $f^*(4,2)=5$.}
\label{fig:mapping:1}
\end{center}
\end{figure}
This guarantees for all $y\in[k'+r-1]$ that at most~$r$ pairs $(i,j)$ are
mapped to~$y$, all of which have the same $j$-coordinate.  In fact only the
first and the last $r-1$ values~$y$ have less than $r$ such pairs mapped
to~$y$, which we call the \emph{exceptional preimages}. Moreover it is easy
to verify that $|f^*(i,j)-f^*(i',j')|\le r$ whenever $|i-i'|\le 1$, that
$f^*(i,j)=f^*(i',j')$ only if $j=j'$, and that $f^*(i,j)\neq f^*(s_i)$ for
all $i,i'\in[k']$ and $j,j'\in[r]$.  Hence $f^*$ is a homomorphism from
$R^r_{k'}$ to $P_{k'+r-1}^r$.

  Our second step is to define the mapping $f^{**}\colon [k'+r-1]\to[m]$ 
  by setting $f^{**}(y):= (y\mod m) + 1$ for all $y\in[k'+r-1]$.
  Clearly $f^{**}$
  is a homomorphism from $P_{k'+r-1}^r$ to $C_m^r$. In conclusion,
  $f\deff f^{**}\circ f^*\circ f'$ is a homomorphism from~$H$ to~$C_m^r$.

  It remains to verify that also $|f^{-1}(i)|\le \tfrac nm(1+\xi)$ for
  every $i\in[m]$.  Indeed, by~\ref{lem:forH:H2} of Lemma~\ref{lem:forH} we
  have $|(f')^{-1}(i,j)| = m_{i,j}\pm \xi' n$ for all $i\in[k'],j\in[r]$.
  Moreover, by construction the preimages of $f^*$ are all of size at
  most~$r$ and only $2(r-1)$ of these preimages, the exceptional preimages,
  are smaller than~$r$. The preimages of~$f^{**}$ are all of the same size
  and $f^{**}$ maps at most one vertex with exceptional preimage under
  $f^*$ to each vertex of~$C_m^r$.  Thus, because $f^{**}\circ f^*$ is a
  mapping from $[k']\times[r]$ to~$[m]$, the preimages of $f^{**}\circ f^*$
  are all of size $\frac{k'r}m \pm r$.  Hence, in total for each
  $i\in[m]$ we have
  \begin{equation*}\begin{split}
      |f^{-1}(i)| 
      &= \left(m_{i,j}\pm \xi' n\right) \cdot \left(\frac{k'r}m \pm r\right) 
      = \Big(\frac{n}{k'r}\pm 1 \pm \xi'n\Big) \cdot \frac{k'r}m\Big(1 \pm
      \frac{m}{k'}\Big) \\
      & \eqByRef{eq:forH:xi} \frac{n}{k'r}\Big(1 \pm \frac{\xi}{2}\Big) \cdot \frac{k'r}m\Big(1 \pm
      \frac{\xi}{3}\Big)
      =\frac nm (1\pm\xi)\,,
  \end{split}\end{equation*}
 where we used $m_{i,j}=\frac{n}{k'r}\pm1$ in the second equality and
  $\frac{m}{k'}\le\frac{m_0}{k'}\le\frac13\xi$ in the third.
\end{proof}

For the proof of Theorem~\ref{thm:app:2} we additionally need the
following lemma, which is implicit in~\cite{AlBrSk10} in the proof of
Theorem~\ref{thm:AlBrSk}. Before we can state this lemma we need some
further definitions.

Assume we are given a complete graph~$K_n$ whose edges are
red/blue-coloured.  Let~$A$ and~$B$ be disjoint vertex sets
in~$K_n$. Then $(A,B)$ is a \emph{coloured $\eps$-regular} pair if
$(A,B)$ is an $\eps$-regular pair in the subgraph of~$K_n$ formed by the
red edges. It is easy to see that such a pair is also $\eps$-regular in
blue.  A vertex partition $(V_i)_{i\in[s]}$ of $V(K_n)$ is called coloured
$\eps$-regular if all but at most $\eps\binom{s}{2}$ of the pairs
$(V_i,V_j)$ with $\{i,j\}\in\binom{s}{2}$ are not coloured $\eps$-regular.
The \emph{coloured reduced graph} $R$ corresponding to this partition is
the graph with vertex set~$[s]$ and an edge for exactly each coloured
$\eps$-regular pair. Each edge $ij$ of~$R$ is coloured in the
majority-colour of the edges of $(V_i,V_j)$. This clearly implies that if
$ij$ is a red edge of~$R$, then the subgraph of $(V_i,V_j)$ formed by the
red edges is $(\eps,\frac12)$-regular.

\begin{lem}[Implicit in~\cite{AlBrSk10}] \label{lem:RamseyPartition} 
  For every $\eps>0$, $r$, and $\tilde m$ there exists $k_0$ and $n_0$ such that
  the following is true for every $n\ge n_0$. Let the edges of $K_n$ be
  red/blue-coloured.
  \begin{enumerate}[label=\abc]
  \item\label{lem:RamseyPartition:a} The graph $K_n$ has a coloured
    $\eps$-regular partition $(V_i)_{i\in[k]}$ with $(2r+3)\tilde m\le k\le k_0$
    and $|V_1|\le|V_2|\le\dots|V_k|\le|V_1|+1$.
  \end{enumerate}
  Let $R$ be the coloured reduced graph corresponding to this partition and
  let $m$ be any multiple of $r+1$ with $k\ge (2r+3)m$.
  \begin{enumerate}[label=\abc,start=2]
  \item\label{lem:RamseyPartition:b} The graph~$R$ contains a monochromatic copy
    of $C_m^r$.
  \end{enumerate}
\end{lem}

We now apply Lemma~\ref{lem:RamseyPartition}, Lemma~\ref{lem:forH:2} and
Theorem~\ref{thm:BK:Blow-up:arr:full} to derive the following result, which
in view of~\eqref{eq:Heawood}, \eqref{eq:bwS}, \eqref{eq:DeltaS}
and~\eqref{eq:aS} directly implies Theorem~\ref{thm:app:2}.

\begin{thm} \label{thm:AlBrSk:ext} 
  Given $a\ge 1$, there exists $n_0$ and $\beta>0$ such that, whenever
  $n\ge n_0$ and $H$ is an $a$-arrangeable $n$-vertex graph with maximum
  degree at most $\sqrt{n}/\log n$ and $\bw(H)\le\beta n$, we have $R(H)\le
  (2\chi(H)+4)n$.
\end{thm}

\begin{proof}
  Let $a$ be given and set $r'\deff a+1$ (and observe that every
  $a$-arrangeable graph is $r'$-colourable). 
  Set $\xi\deff 1/(100r')$. Choose~$\eps$ as given by
  Theorem~\ref{thm:BK:Blow-up:arr:full} with $C\deff 0$, $a$, $\Delta_R\deff 2r'+1$,
  $\kappa\deff 2$ and $\delta'\deff 1/4$, $c\deff 1$. If necessary decrease $\eps$ such
  that $\eps\le\xi/(4r')$. 
  Further set $\tilde m\deff 100r'^2$.  Let
  $n'_0$ and $k_0$ be as returned by Lemma~\ref{lem:RamseyPartition} for
  these $\eps,r',\tilde m$. Then continue the application of
  Theorem~\ref{thm:BK:Blow-up:arr:full} with $s:=k_0$ and obtain~$n_0''$.
  Set $m_0\deff k_0$ and
  \begin{equation*}
    n_0\deff\max\{n_0'\,,\, n_0''\,,\, 100 m_0r'\}\,.
  \end{equation*}
  Let $\beta>0$ be as given by Lemma~\ref{lem:forH:2} with parameters
  $\xi$, $r'$, and $m_0$.  Finally, let~$n$ and~$H$ be given, set
  $r:=\chi(H)$, and assume we have a red/blue-colouring of the edges of
  $K_{(2r+4)n}$.

  Lemma~\ref{lem:RamseyPartition}\ref{lem:RamseyPartition:a} asserts that
  there is a coloured $\eps$-regular partition $(V'_i)_{i\in[k]}$
  of~$K_{(2r+4)n}$ with $(2r+3)\tilde m\le k\le k_0$ whose clusters differ
  in size by at most~$1$.  Let~$R'$ be the coloured reduced graph of the
  partition $(V'_i)_{i\in[k]}$. Let~$m$ be the multiple of $r+1$ which
  satisfies $(2r+3)m \le k < (2r+3)(m+r+1)$.  Observe that this and $k\ge
  (2r+3)\tilde m$ implies $m\ge\tilde m-r$ and thus 
  \begin{equation}
  \label{eq:Ramsey:m}
    \frac12 m\ge\frac12(\tilde m-r)\ge 2r^2+5r+3
  \end{equation}
 because $\tilde m= 100r'^2\ge 100r^2$. Further, $m\le k\le k_0=m_0$ and so 
  \begin{equation}
  \label{eq:Ramsey:mn}
    \frac mn\le \frac{m_0}{n_0}
    \le\frac{1}{100 r'} \le \frac{1}{100 r}\,.
  \end{equation}
  We conclude that we have
  \begin{equation*}\begin{split}
    |V'_i|
    &\ge\frac{(2r+4)n}{k}-1
    \ge\frac{(2r+4)n}{(2r+3)(m+r+1)}-1
    \geByRef{eq:Ramsey:m} \frac{(2r+4)n}{(2r+3.5)m}-1 \\
    &= \Big(1+\frac{0.5}{2r+3.5}-\frac mn\Big)\frac{n}{m}
    \geByRef{eq:Ramsey:mn} \Big(1+\frac{1}{20r}-\frac{1}{100r}\Big)\frac{n}{m}
    \ge (1+2\xi)\frac{n}{m}
  \end{split}\end{equation*}
  because $\xi=1/(100r')\le 1/(100r)$.  In addition, by
  Lemma~\ref{lem:RamseyPartition}\ref{lem:RamseyPartition:b} there is a
  monochromatic $C_m^r$ in $R'$, without loss of generality a red
  $C_m^r$. Let $U\subseteq V(K_{(2r+4)n})$ be the set of all vertices contained in
  clusters of this~$C_m^r$.

  Our next step is to apply Lemma~\ref{lem:forH:2} to the graph~$H$ with
  parameters $\xi$, $r'$, $m_0$, $\beta$, $r$ and $m$.  This lemma
  guarantees a homomorphism $f\colon H \to C_m^r$ with $|f^{-1}(i)|\le
  (1+\xi) \tfrac nm$ for every $i\in[m]$.  By setting $W_i\deff f^{-1}(i)$
  we obtain a partition $(W_i)_{i\in V(C_m^r)}$ of~$H$.
  
  We finish the proof with an application of
  Theorem~\ref{thm:BK:Blow-up:arr:full}.
  %
  %
  In this application we will not have image restricted vertices and we
  will use $R:=C_m^r$. Observe that $\Delta(R)=2r<\Delta_R$ and thus
  \ref{item:Blow-up:R} of Theorem~\ref{thm:BK:Blow-up:arr:full} is
  satisfied. The partition $(W_i)_{i\in V(C_m^r)}$ and the conditions
  on~$H$ guarantee that also condition~\ref{item:Blow-up:H} of
  Theorem~\ref{thm:BK:Blow-up:arr:full} is satisfied.

  Now let $G'$ be the subgraph of $K_n$ with vertices~$U$ and all red edges
  of~$K_{(2r+4)n}$ in~$U$. In the following we consider this graph as an uncoloured
  graph.  Clearly the partition $(V'_i)_{i\in[k]}$ induces a partition
  $(V'_i)_{i\in V(C_m^r)}$ of $G'$ which is $(\eps,\frac12)$-regular
  on~$C_m^r$.  Moreover, since~$C_m^r$ has maximum degree~$2r$, by deleting
  from each of these clusters $V'_i$ at most $2r\eps|V'_i|\le\frac12\xi|V'_i|$
  vertices we can obtain a partition $(V_i)_{i\in V(C_m^r)}$ of a
  subgraph~$G$ of~$G'$ which is $(\eps,\frac 14)$-super-regular on~$C_m^r$
  and satisfies $|V_i|\ge(1+\xi)\frac{n}{m}\ge|W_i|$. Hence for~$G$ and
  $(V_i)_{i\in V(C_m^r)}$ also condition~\ref{item:Blow-up:G} of
  Theorem~\ref{thm:BK:Blow-up:arr:full} is satisfied.

  Thus Theorem~\ref{thm:BK:Blow-up:arr:full} implies that there is a copy
  of~$H$ in~$G$. This copy corresponds to a red copy of~$H$ in the
  red/blue-coloured~$K_{(2r+4)n}$.
\end{proof}

\section{Concluding remarks}
\label{sec:concl}

\paragraph{Optimality of Theorem~\ref{thm:BolKom:arr}.}

The degree bound $\Delta(H)\le\sqrt{n}/\log n$ in
Theorem~\ref{thm:BolKom:arr} arises from our proof method: For the Blow-up
Lemma, Theorem~\ref{thm:BK:Blow-up:arr:full}, such a degree bound is necessary
(see~\cite[Proposition~35]{BKTW_blowup}). For trees~$H$, however, 
the corresponding result of Koml\'os, Sark\"ozy, and Szemer\'edi~\cite{KSS_trees} 
requires only the weaker condition $\Delta(H)=o(n/\log n)$.
It is thus well possible that our maximum degree condition is not best
possible and could be improved to $o(n/\log n)$.

\paragraph{Blow-up Lemmas.}

In the original formulation of the Blow-up
Lemma~\cite{KSS97,KSS98,RodlRuci99} the regularity~$\eps$ required for the
super-regular pairs depends on the number of clusters~$k'$ used in an
application. Consequently, this lemma can never
be used on the \emph{whole cluster graph} obtained from an application of
the Regularity Lemma: the number of clusters~$k$ the Regularity Lemma
produces depends on the required regularity~$\eps$. Moreover, all pairs
used in the embedding have to be super-regular.

The Blow-up Lemma for arrangeable graphs formulated in~\cite{BKTW_blowup}
overcomes the first difficulty: Here~$\eps$ only depends on the maximum degree
of the reduced graph of the super-regular partition that is used.
(In fact, fairly straight-forward modifications of the original Blow-up
Lemma proof from~\cite{KSS97} would also allow for a corresponding result
for bounded degree graphs.)

In Theorem~\ref{thm:Blow-up:arr:mixed} we also overcome the second
difficulty: Pairs into which we only want to embed few edges are now
allowed to be merely $\eps$-regular.
This allows us to avoid the occasionally tedious procedure of setting up
suitable image restrictions and then applying the Blow-up Lemma several times.
This might turn out could be useful for other applications as well.

\paragraph{Degeneracy.}

Though by now many important graph classes were shown to be $a$-arrangeable
for some constant~$a$, the notion of arrangeability has the disadvantage of
seeming somewhat artificial at first sight. The notion of degeneracy is
more natural (and more general):
A graph~$H$ is $d$-degenerate if there is an ordering of its vertices such
that each vertex has at most~$d$ neighbours to its left.

It would be very interesting to obtain an
analogue of Theorem~\ref{thm:BolKom:arr} for $d$-degenerate graphs.
However, most likely this problem is very hard. Indeed, a version of the Blow-up
Lemma for $d$-degenerate graphs would imply the difficult and long-standing
Burr-Erd\"os conjecture~\cite{BurrErdos}, which states that degenerate
graphs have linear Ramsey number.

\paragraph{Acknowledgement.}

The authors thank Peter Allen for suggesting Theorem~\ref{thm:app:2}.

\bibliographystyle{siam}   
\bibliography{SpanningEmbeddings}  

\end{document}